\documentclass[12pt,oneside]{amsart}
\usepackage{amssymb, amscd, amsmath, amsthm, latexsym, hyperref}
\usepackage{pb-diagram, fancyhdr, graphicx, psfrag}
 \usepackage[text={6.3in,8.5in},centering,letterpaper,dvips]{geometry}
\newcommand{\compactlist}{\begin{list}{$\bullet$}{\setlength{\leftmargin}{1em}}}

\def\qq{{\bf Q}}

\def\calc{\mathcal{C}}

\def\cs{\mathbin{\#}}
\DeclareMathOperator{\cfk}{\it CFK}
\newcommand{\spinc}{\ifmmode{{\mathfrak s}}\else{${\mathfrak s}$\ }\fi}
\newcommand{\spinct}{\ifmmode{{\mathfrak t}}\else{${\mathfrak t}$\ }\fi}

\newcommand{\fig}[2] { \includegraphics[scale=#1]{#2} }
\def\U{\Upsilon}

\newtheorem*{utheorem}{Theorem}
\newtheorem{theorem}{Theorem}
\newtheorem{lemma}[theorem]{Lemma}

\theoremstyle{definition}
\newtheorem{definition}[theorem]{Definition}

\begin{document}
\title{The four-genus of connected sums of torus knots}
\author{Charles Livingston and Cornelia A. Van Cott}
\thanks{This work was supported in part by the Simons Foundation and a grant from the National Science Foundation.}

\address{ }
\email{}
\address{Charles Livingston: Department of Mathematics, Indiana University, Bloomington, IN 47405 }
\email{livingst@indiana.edu}
\address{Cornelia A. Van Cott: Department of Mathematics, University of San Francisco, San Francisco, CA  94117}
\email{cvancott@usfca.edu}


\begin{abstract} We study the four-genus of linear combinations of torus knots:  $g_4(aT(p,q) \cs -bT(p',q'))$.  Fixing positive $p,q,p',$ and $ q'$, our  focus is on the behavior of the  four-genus as a function of positive $a$ and $b$.  Three types of examples are presented: in the first, for all  $a$ and $b$  the four-genus is completely determined by the Tristram-Levine signature function; for the second,  the recently defined Upsilon function of Ozsv\'ath-Stipsicz-Szab\'o determines the four-genus for all $a$ and $b$; for the third,   a  surprising interplay between signatures and Upsilon appears.   

 \end{abstract}

\maketitle

\section{Introduction} 

The four-genus of torus knots was determined in Kronheimer-Mrowka's resolution of  the Milnor and Thom Conjectures~\cite{km1}; in brief, $g_4(T(p,q)) = g_3(T(p,q)) = (p-1)(q-1)/2$.  The $\tau$--invariant of Ozsv\'ath-Szab\'o~\cite{os1} and the $s$--invariant of Rasmussen~\cite{ras} provided alternative approaches to the study of the four-genus of knots. Both  offer an immediate generalization: for any collection of {\it positive} torus knots, $g_4(\cs T_i) = \sum g_3(T_i)$.

In contrast to these results, the   four-genus of the differences of   positive torus knots, $g_4( T(p,q) \cs - T(p',q'))$, is largely unknown, even though it arises naturally in classical knot theory, for instance in studying unknotting sequences of knots and the Gordian distance between knots.   This problem of determining this four-genus also appears in the study of deformations of algebraic curves and in determining the minimal cobordism distance between torus knots~\cite{ baader1, baader2, borodzik-hedden, borodzik-livingston,  feller1, feller2, owens-strle}.

Here we will  consider a more general problem, determining  $g_4( aT(p,q) \cs -bT(p',q'))$;  we will always restrict our attention to the open case, in which all the parameters are positive.  Our principal goal is to explore the complementary nature of two of the strongest invariants that bound the four-genus: the classical Tristram-Levine signature function, $\sigma_K(t)$, defined in~\cite{levine, tristram}, and the Ozsv\'ath-Stipsicz-Szab\'o Upsilon invariant, $\U_K(t)$, defined in~\cite{oss}.  (Note that the signature function is  determined by the Milnor signatures~\cite{milnor1}, and the Upsilon bounds are determined by Heegaard-Floer bounds discovered by Hom and Wu~\cite{hom-wu}.  The Upsilon function generalized the $\tau$--invariant: $\tau(K) = \U_K(t)/t $ for small $t$.)

We give several positive results in which either $\sigma_K(t)$ or $\U_K(t)$ singlehandedly determines the four-genus of a subfamily of knots of the form $aT(p,q) \cs -bT(p',q')$, and we also give some results where neither of the two invariants alone determines the four-genus, but together they are sufficient. Finally, we identify large families of such differences of torus knots for which the determination of the four-genus is inaccessible with only these two invariants.  Our approach offers a new perspective from which to view the limits of current techniques and to identify further  challenging problems. 

Three theorems illustrate the nature of our results.  The first is an unpublished theorem of Litherland which shows the strength of signatures; we prove this result in Section~\ref{section:litherlandthem}, along the lines of Litherland's proof.  The second is a theorem which uses only the Upsilon invariant; this is proved in Section~\ref{section-purely-upsilon}.  The third theorem uses both the signature and Upsilon invariants; this result is proved in Section~\ref{mixed-example}.

\begin{theorem}\label{thm:litherland}  Let $K = aT(2,2k+1)\cs  -bT(2,2j+1)$.  Then 
$$g_4(K)  = \max_{t\in[0,1]}( |\sigma_K(t)| /2).$$
\end{theorem}

\begin{theorem}\label{theorem:pqr} Let $K = aT(p,qr)\cs  -bT(q,pr)$ with    $p<q$ and $r<\frac{q}{q-p}$.   Then 

$$ g_4(K) =  
\begin{cases}
| \tau(K)  |      & \mbox{\ if\ } a \le b \\
  \frac{p}{2} | \U_K(\frac{2}{p})|   & \mbox{\ if\ } a \ge b.  
 \end{cases}
   $$
   Furthermore, if $a= b$, then $  \tau(K) = -  \frac{p}{2} | \U_K(\frac{2}{p})|$; if $a >b$, then $  |\tau(K)| <  \frac{p}{2} | \U_K(\frac{2}{p})|$.\end{theorem}

\begin{theorem}\label{example3438} Let $K = aT(3,4)\cs -bT(3,8)$. Then 
$$g_4(K) =  
\begin{cases} 
\max_{t\in (0,1]}\frac{1}{t}| \U_K(t)|    =7b-3a   &\mbox{if } 0\leq a< 2b \\
\max_{t\in [0,1]}\frac{1}{2}| \sigma_K(t)|   =3a - 5b    & \mbox{if } a \geq  2b.
\end{cases}  $$

\end{theorem}
\vspace{.3cm}
The results of Theorem~\ref{theorem:pqr} are of particular note because the theorem utilizes $\U_K(t)$ at values of $t$ strictly between $0$ and $1$. As mentioned previously, for $t$ close to 0, $\U_K(t)$ is simply equal to $\tau(K)t$. When $t = 1$, the invariant $\U_K(1)$ is denoted by $\upsilon(K)$ and has been used in several ways. Ozsv\'ath-Stipsicz-Szab\'o ~\cite{oss2} recently showed that  $\upsilon(K)$  provides bounds on the four-dimensional crosscap number of a knot. The invariant $\upsilon(K)$ has also been used in~\cite{feller2} to provide a sharp bound on the four genus of $T(p,q)\#-T(p',q')$ for small values of $p$ and $p'$. The {\it alternating number} of torus knots was studied using $\upsilon(K)$ in~\cite{fpz}.  Theorem~\ref{theorem:pqr}  illustrates that for specific pairs $T(p,q)$ and $T(p',q')$, for different values of $a$ and $b$, the best bound on the four-genus is achieved from either $\tau(K)$, $\upsilon(K)$, or $\U_K(t)$ for some value of $t$ strictly between $0$ and $1$.

 \subsection{The stable four-genus of knots}  Many of our examples are most easily illustrated in terms  of the {\it stable  four-genus of knots}, defined in~\cite{livingston1}:
 $$g_s(K) = \lim_{n\to \infty} g_4(nK).$$
 The stable four-genus extends to give a semi-norm on the  tensor product of the smooth concordance group with the rational numbers, $\calc \otimes \qq$. Letting $t = a/(a+b)$, the statement that $g_4(aK \cs-bJ) = g$ quickly implies $g_s(tK \cs (1-t)J) \le g/(a+b).$  Figure~\ref{fig:3438} illustrates the lower bounds  on the stable genus $g_s(tT(3,4) \#- (1-t)T(3,8))$  that are provided by the signature function (marked with thinner segments, drawn in red) and the Upsilon function (marked with thicker segments, drawn in blue).  These computations are  presented in detail    in Section~\ref{mixed-example}.

 Theorem~\ref{example3438} states that the four-genus for this particular family of knots is exactly determined  by the larger of these lower bounds.  In Section~\ref{generalizing} we will illustrate other examples from the perspective of the stable genus, but in every case, the results immediately transfer back to give precise results concerning the four-genus.
 
\begin{figure}[h]
\fig{.45}{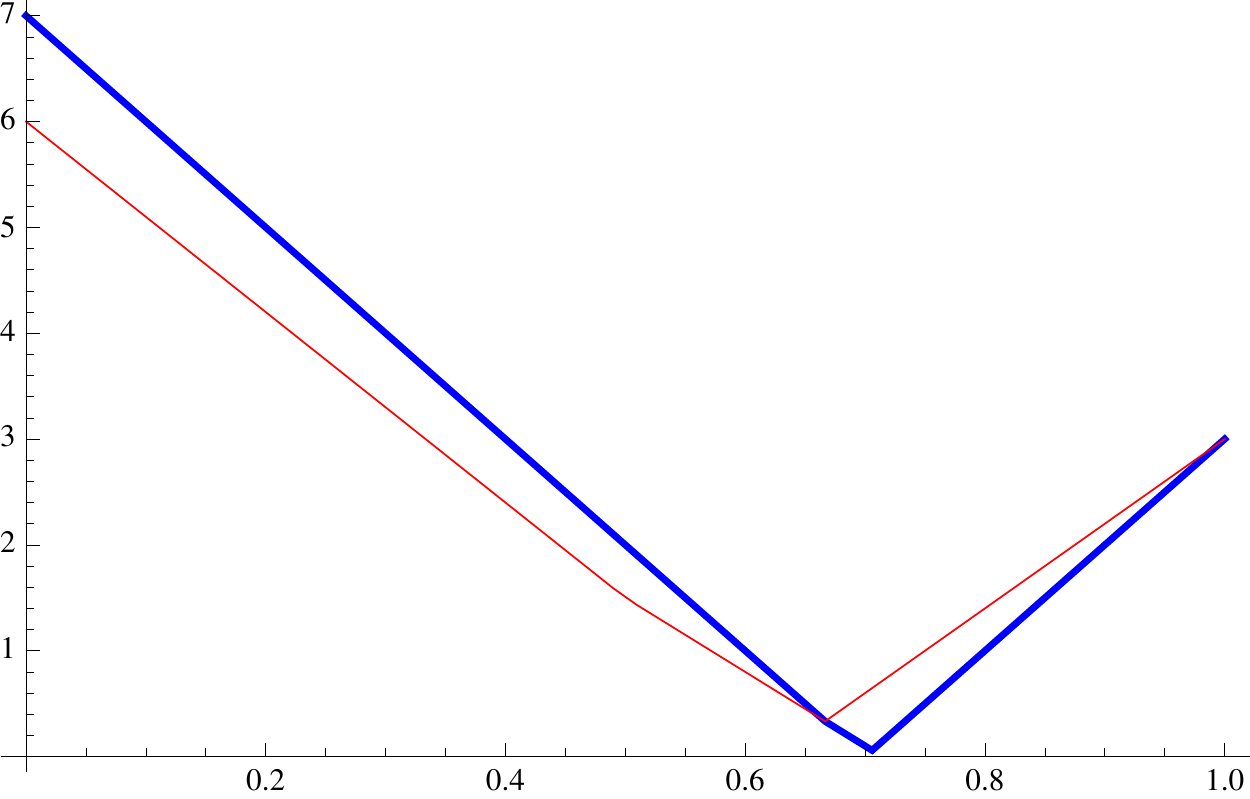}  \caption{The Upsilon lower bound (thick blue line) and the signature lower bound (thin red line) on the stable genus of $K =tT(3,4) - (1-t)T(3,8)$ where $t \in [0,1]$. The two lines meet at $t = \tfrac{2}{3}$. }\label{fig:3438}
\end{figure}

 \subsection{Outline}  
 In Sections~\ref{section:signatures} and~\ref{section:upsilon} we review the definitions and basic properties of the Levine-Tristram signature function and the Upsilon function of Ozsv\'ath-Stipsciz-Szab\'o.   Section~\ref{section:stable} reviews the stable four-genus.  Sections~\ref{section:litherlandthem},~\ref{section-purely-upsilon}, and ~\ref{mixed-example} present the proofs of  Theorems~\ref{thm:litherland},~\ref{theorem:pqr}, and~\ref{example3438}, respectively, along with generalizations.  Section~\ref{generalizing} explores the limits of our techniques, giving examples where the four-genus of $aT(p,q) \cs -bT(p',q')$ is still unknown. In Appendix~\ref{appendix:torusknotU} we compute  the second singularity of the Upsilon function for torus knots.  
 
 \subsection{Acknowledgments}  Since we first posted this work, Peter Feller and Allison Miller have informed us that they have found new examples of families of torus knots for which the four-genus can be determined.  Those examples are described in the closing section of this paper. We appreciate the feedback we received from them.
 

\section{Review of Signatures}\label{section:signatures}

 If $V_K$ is a Seifert matrix for an  oriented link $L$, there is an associated Hermitian matrix $$V_L(t)= (1- \omega)V + (1-\overline{\omega}) \text{transpose}(V),$$ where  $\omega = e^{ \pi i t}$, $0\le t \le 1$.  For each $t\in [0,1)$, we   define the signature function $\sigma_L(t)$ to be right-sided limit of the signature  of this matrix at $t$.   Usually the signature function is defined in terms of the average of the two-sided limits; either approach yields a concordance invariant, and both provide  identical bounds on the four-genus.  The advantage of using the one-sided limit is that it ensures that the maximum and minimum values  occur at discontinuities of the function, which are  values of  $t$ for which $e^{\pi i t}$ is a root of the Alexander polynomial.

For knots, the signature is an even  integer-valued step function.  As an example, in Figure~\ref{sig-pdf} we illustrate $\sigma_K$ of the torus knots $T(3,11)$ (below the axis) and $-T(5,6)$ (above the axis).  In the figure, each step is of size either 2 or $-2$.

\begin{figure}[h]
\fig{.5}{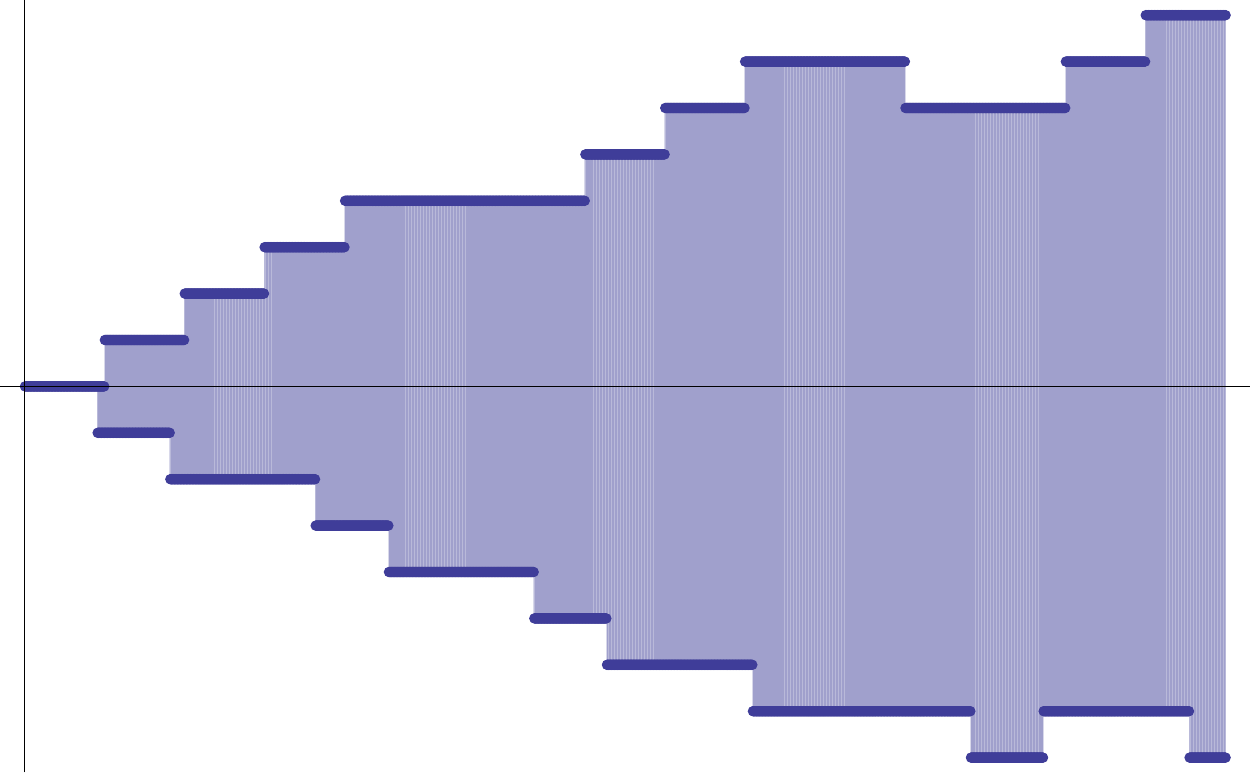}  \caption{The signature functions of $T(3,11)$ (below axis) and $-T(5,6)$ (above axis). }
\end{figure}\label{sig-pdf}

The signature function bounds the four-genus:   $g_4(K) \ge \frac{1}{2} |\sigma_K(t)|$,  for all $t$.   For instance, from the figure we see that $g_4(T(3,11)) \ge 8$ and $g_4(T(5,6)) \ge 8$ (the actual four-genus is 10 for both these knots).    Figure~\ref{figsigdif} illustrates the sum of these two signature functions, that is $\sigma_{T(3,11) \# - T(5,6)}(t)$, from which we get the bound $g_4(T(3,11) \# - T(5,6)) \ge 2$.

\begin{figure}[h]
\fig{.5}{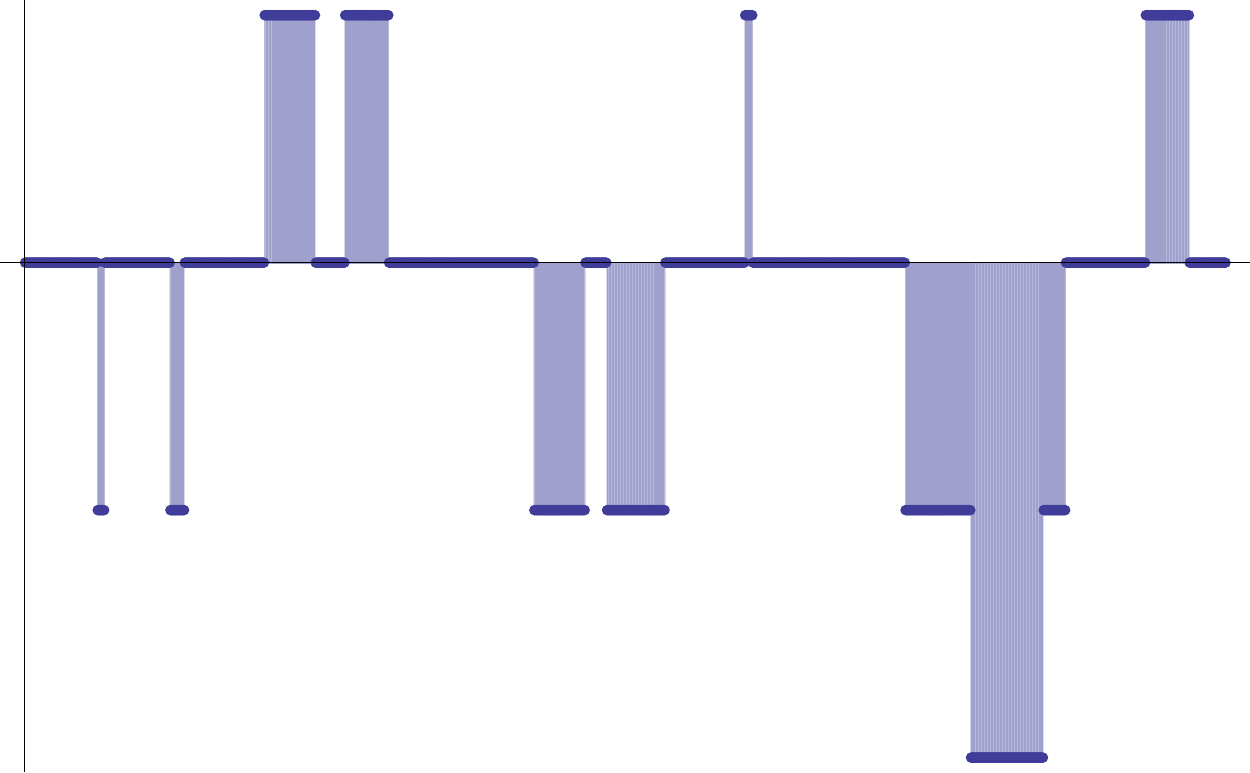}  \caption{The signature function of $K = T(3,11) \# - T(5,6)$.}\label{figsigdif}
\end{figure}

The necessary ingredients for computing the signature function for torus knots are describe in~\cite{goldsmith, litherland}.  We will also need to understand the signature of $(2,2k)$ torus links; these are also discussed in~\cite{goldsmith}, along with the genus bounds we are using.


\section{The Upsilon function, $\U_K(t)$}\label{section:upsilon}

The knot concordance invariant $\U_K(t)$ was first introduced in~\cite{oss}.  In an intuitive sense, it captures certain aspects of the  shape  of the knot Heegaard Floer complex $\cfk^\infty(K)$.  We refer the reader to~\cite{oss} for further details, or to the more expository account~\cite{livingston2}.  For torus knots, $\U_K(t)$ is easily computed from the Alexander polynomial or directly from the semigroup generated by $p$ and $q$, with algorithms presented in~\cite{borodzik-livingston, oss}.  Key results concerning $\U_K(t)$ include:

\begin{itemize}

\item  $\U_K(t)$Ä is  piecewise linear  with domain $[0,2]$ and $\U_K(0) = 0$.\vskip.05in

\item The map $K \to  \U_K $ defines a homomorphism from the smooth concordance group to the group of continuous functions on $[0,2]$.\vskip.05in

\item     $\U_K(t) = \U_K(2-t)$, which permits us to focus solely on the interval $[0,1]$.\vskip.05in

\item  For all $t \in (0,1]$, $g_4(K) \ge | \frac{\U_K(t)}{t}|$.\vskip.05in

\item There is a $t_1>0$ such that for all $t \in (0,t_1]$ one has $ {\U_K(t)}/{t} = \tau(K)$, where $\tau(K)$ is the Ozsv\'ath-Szab\'o $\tau$-invariant, defined in~\cite{os1}. 

\end{itemize}   

As an example,  Figure~\ref{ups-pdf}  illustrates the graph of $\U_K(t)$ for $K= T(5,6) \# - T(3,11)$.  The second, larger, function in the figure  is $\U_K(t)/t$.  The functions have domain $[0,1]$ and the maximum value  of $\U_K(t)/t$ is 2, attained at $t = \frac{2}{3}$. The other singular points are $t = \frac{2}{5}$ and $t = \frac{4}{5}$.

\vskip.1in
\begin{figure}[h]
\fig{.5}{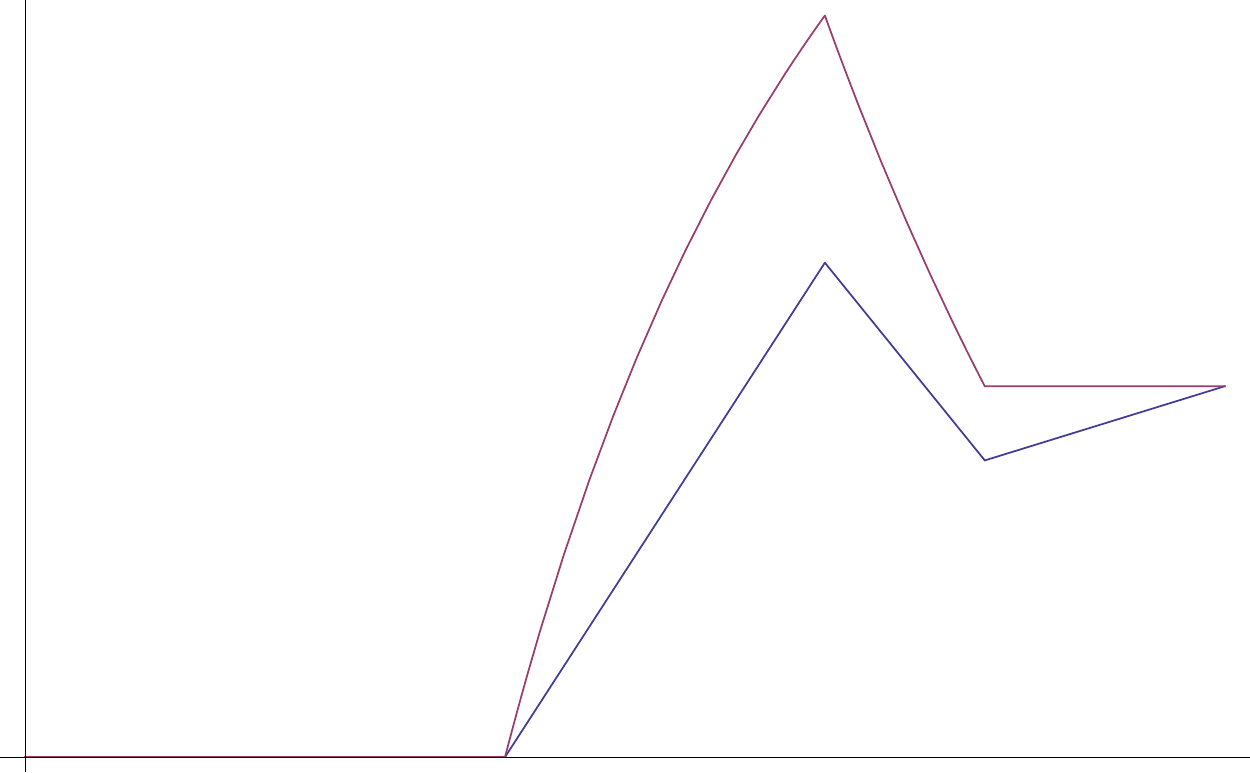}  \caption{$\U_K(t)$ and $\U_K(t)/t$ for $K= T(5,6) \# - T(3,11)$ for $t \in [0,1]$. }\label{ups-pdf}
\end{figure}

Since $\U_K(t)/ t$ is not piecewise linear, there is the possibility that finding maximum values could be complicated.  Given its form however, the following result is a simple  exercise.

\begin{theorem} For a knot $K$, the maximum value of $|\U_K(t)|/t$ is attained either at a singular point of the derivative  $\U'_K(t)$  on the interval $(0, 1)$ or at $t = 1$.
\end{theorem}


\section{The Stable Four-Genus}\label{section:stable}
 In the introduction we described the stable four-genus.  Here we present a few more details.
 
The stable four-genus of a knot is defined to be $$g_s(K) = \lim_{n\to \infty} g_4(nK)/n.$$
In~\cite{livingston1} it is observed that this limit is well-defined. If we let $\calc$ denote the concordance group, then $g_s$ induces a semi-norm on the rational vector space $\calc \otimes \qq$.  That is, $g_s(aK) = ag_s(K)$ and $g_s(K \cs J) \le g_s(K) + g_s(J)$.  It is unknown whether $g_s$ is a norm:  there may be nontrivial elements $K \in \calc_\qq = \calc \otimes \qq$ with $g_s(K) =0$.

If $\nu(K)$ is any {\it additive} function on $\calc$ (meaning $\nu(nK) = n\nu(K)$ for $n\ge 0$)  it extends linearly to $\calc_\qq$.  Thus, we have the following.

\begin{theorem}If $\nu$ is an additive function on $\calc$ satisfying $|\nu(K)| \le g_4(K)$ for all $K$, then $|\nu(K)| \le g_s(K) $ for all $K \in \calc_\qq$.

\end{theorem}

For our work here, we will use $\U_K$ and $\sigma_K$ to find lower bounds on $g_4(K)$, but since the bounds arise from homomorphisms,  they  provide identical bounds on $g_s(K)$.  On the other hand, our realization results apply equally for the  four-genus as for the stable four-genus: for each  $K \in \calc_\qq$ for which we  compute $g_s(K)$ exactly, it follows that  for any multiple $aK \in \calc$, $g_4(aK) = ag_s(K)$.  This holds because of the explicit knots that we use to  attain our realization results; more precisely, we realize our bounds with knots of the form $aT(p,q) -bT(p',q')$ where $a$ and $b$ are relatively prime.

 
\section{Litherland's theorem on $(2,k)$--torus knots}\label{section:litherlandthem}

\subsection{Statement of theorem and preliminaries}$\ $


\begin{theorem}\label{lith} Let $K = aT(2,k)\cs  -bT(2,j)$ where $a$, $b$, $k$, and $j$ are positive integers and $k$ and $j$ are odd.  Then 
$$g_4(K)  = \max_{t\in[0,1]}( |\sigma_K(t)| /2).$$
\end{theorem}

As described in Section~\ref{section:signatures}, we are working with  the signature function $\sigma_K(t)$ defined to equal  the standard signature function  except at the values of $t$ where $\sigma_K(t)$ is discontinuous; at discontinuities,  the value is given as a right-handed limit rather than as the average of the two-sided limits.  The benefit  of this approach is the signature function is right continuous.  Regardless of the choice of convention, the  maximum and minimum values of signature functions are equal; these are exactly the values of the signature function which we are after. Moreover,  with this modification, we have the following result.

\begin{lemma}Let $K$ be any knot. If $\sigma_K(t)$ is the right continuous signature function, then the maximum and minimum values of the signature are each attained at one of the function's points of discontinuity. 
\end{lemma}

In the special case where $K = T(2,k)$, the points of discontinuity for the signature function and then its values are readily computed.  We give a summary of the derivation.

\begin{lemma}
$  $
\begin{enumerate}
\item
The points of discontinuity in $\sigma_{T(2,k)}(t)$ occur at the values in the set  $$X_k =\{ i/k \ |\   0 \le i <k  \text{\ and\ }   i \equiv k \mod 2\}.$$
\item For $t\in[0,1)$, the right continuous function $\sigma_{T(2,k)}(t)$ is given by 
$$\sigma_{T(2,k)}(t) =-( \lfloor  kt \rfloor_k +1),$$
where $ \lfloor  x  \rfloor_k $ denotes the greatest integer $n$ such that $n \equiv k \mod 2$ and $n \le x$.
\end{enumerate}
\end{lemma}
 \begin{proof}
The Alexander polynomial, which can be computed from the standard Seifert matrix $V$ as $\text{ det}(V - t\text{ transpose}(V))$, is $ (t^k + (-1)^{k+1})/(1+t)$, which has $k - 1$ roots on the unit circle, each with multiplicity one.  This yields statement (1).

Since the roots each have multiplicity one, the signature function jumps by $\pm 2$ at each of these roots.  The signature at $1$ is given as the signature of   $(V + \text{transpose}(V))$, which can be computed  to be $-k$.  Near 0, the signature function is either 0 or $-1$, depending on whether $k$ is odd or even, respectively.  Given the number of jumps, each must be negative, yielding the desired result.

\end{proof}

In light of the above discussion, it follows that for any torus knot $T(2,k)$, 
$$\max_{t\in[0,1]}( |\sigma_{T(2,k)}(t)|) = \max_{t\in X_k}(| \lfloor  kt \rfloor_k +1|) .$$
With this motivation, we introduce the following notation.  (Notice that the following notation entails a change of sign in the signature function; this somewhat simplifies our computations.)

\begin{definition}$\ $
$\sigma_k(t) =  \lfloor  kt  \rfloor_k +1$. \vskip.05in
\end{definition}

For the linear combination of torus knots $K = aT(2,k)\cs  -bT(2,j)$, we have a corresponding observation: 
\begin{eqnarray*}
\max_{t\in[0,1]}( |\sigma_{K}(t)|) &=& \max_{t\in[0,1]}( |a\sigma_{T(2,k)}(t)- b\sigma_{T(2,l)}(t)|)\\
&=& \max_{t\in X_k\cup X_l}( |a\sigma_{k}(t)- b\sigma_{l}(t)|).
\end{eqnarray*}
This motivates the following notation: 
\begin{definition}$\ $
\begin{enumerate}
\item   $\sigma_{a,k,b,l} (t)  = a\sigma_k (t) - b\sigma_l (t) $.\vskip.05in
  \item  $\mu(a,k,b,l) =   \max_{t\in X_k\cup X_l}(| \sigma_{a,k,b,l} (t)| )$.\vskip.05in
\end{enumerate}
\end{definition}
Thus to prove Theorem~\ref{lith}, it suffices to prove that 
$$g_4(K) = \tfrac{1}{2}\mu(a,k,b,l).$$

\subsection{Key steps of proof of Theorem \ref{lith}}$\ $

The proof of Theorem~\ref{lith} is recursive and requires one to consider cases in which one, but not both, of $k$ and $l$ are even.  In this case, $K = aT(2,k)\cs  -bT(2,l)$ is the connected sum of a torus knot and a torus {\em link}. For such $K$, it is simpler to  consider the first Betti number rather than genus. 

\begin{definition}
$\beta(K) = \min \{ \beta_1(F) \ |\  F\subset B^4 \text{ and } \partial F = K\}$.
\end{definition}

Theorem~\ref{lith} can be reformulated as follows: 

\begin{theorem} Let $K = aT(2,k)\cs  -bT(2,l)$. If at least one of $k$ and $l$ are odd, then   $$\beta(aT(2,k) \# - bT(2,l)) = \mu(a,k,b,l).$$\vskip.05in
\end{theorem} 


We already have that $\beta(aT(2,k) \cs - bT(2,l) ) \ge \mu(a,k,b,l)$.  The proof of equality readily follows from the next two recursive results.

\begin{lemma}\label{lith1lemma}  Write $k = ql +r$ with $0\le r <l$.
\begin{enumerate}
\item If $b \le qa$, then $$\beta(a T(2,k) \cs- bT(2,l) )\le a(k-1) - b(l-1).$$\vskip.05in

\item If $r= 0 $ and $b > qa$, then $$\beta( aT(2,k)\cs - bT(2,l) )\le -a(k - 2q+1) + b(l-1).$$\vskip.05in

\item If $r \ne 0$ and $b> qa$, then $$\beta( aT(2,k) \cs- bT(2,l) )\le \beta((b-qa)T(2,l) \cs- aT(2,r)) + qa.$$\vskip.05in
\end{enumerate}

\end{lemma}

\begin{lemma}\label{lith2lemma} Write $k = ql +r$ with $0\le r <l$.
\begin{enumerate}
\item If $b \le qa$, then $$\mu(a,k,b,l) = a(k-1) - b(l-1).$$\vskip.05in

\item If $r= 0 $ and $b > qa$, then $$\mu(a,k,b,l) = -a(k - 2q+1) + b(l-1).$$\vskip.05in

\item If $r \ne 0$ and $b> qa$, then $$\mu(a,k,b,l) \ge  \mu(b-qa,l,a,r) + qa.$$\vskip.05in
\end{enumerate}

\end{lemma}


\subsection{Proof of Lemma~\ref{lith1lemma}} $\ $

The proof of statement (1) is the most straightforward. We consider $K = a T(2,k) \cs- bT(2,l)$, where $k = ql+r$ and $0\leq r <l$ and $b \leq aq$. 

Each Seifert surface for $T(2,k)$ contains a connected sum of $q$ copies of $T(2,l)$ bounding a subsurface.  Thus, the Seifert surface for $aT(2,k)$ contains the connected sum of $aq$ copies of $T(2,l)$ bounding a subsurface.  Therefore, since $b \le aq$, we can surger the canonical surface for $K = aT(2,k) \cs -bT(2,l)$, replacing a subsurface bounded by $bT(2,l) \cs -bT(2,l)$  of genus $b(l-1)$ with a slice disk. The resulting surface has first Betti number
$$a(k-1) +b(l-1) - 2b(l-1) = a(k-1) - b(l-1),  $$
as desired.

The proof of statement (2) is similar.  Now we can surger out Seifert surfaces for $\cs_{aq} (T(2,l) \cs-T(2,l))$.  These surfaces have total genus $aq(l-1)$.  Thus, the first Betti number  of the resulting surface in the 4--ball is $$a(k-1) + b(l-1) - 2aq(l-1)= -a(k-2q+1) + b(l-1),$$
as desired.

For statement (3), we again begin with $K = a T(2,k) \cs- bT(2,l)$, but this time we proceed differently from before. Instead of surgering subsurfaces from a Seifert surface for $K$, we perform a sequence of band moves to $K$. In particular, we can perform $aq$ band moves on $K$ so that each copy of $T(2,k)$ is connected to $q$ different copies of $-T(2,l)$. The resulting knot (or link) is $L = (b - qa)T(2,l)\cs - T(2,r)$. Therefore $K$ bounds a surface with first Betti number $\beta(L) + aq$, as desired.

\subsection{Proof of Lemma~\ref{lith2lemma}} $\ $ 

Before beginning the proofs of the three statements, note that using Lemma~\ref{lith1lemma} and the fact that $\beta(K) \ge \mu(K)$, we need prove only  inequalities, rather than equalities.\vskip.05in


\noindent{\bf Case (1)} For the proof of statement (1), we need to show that if $b \le qa$, then $$\mu(a,k,b,l) \ge a(k-1) - b(l-1).$$  
For this particular case, we can simply compute the value of the signature function $\sigma_K(t)$ at $t=1$ (called the {\em Murasugi signature}). This value is known to be $$\sigma_K(1) = a\sigma_{T(2,k)}(1) - b\sigma_{T(2,l)}(1) =  a(k-1) -b(l-1),$$ giving the desired bound:
$$ \mu(a,k,b,l) = \max_{t\in[0,1]}( |\sigma_K(t)|) \geq a(k-1) -b(l-1).$$
\vskip.05in

\noindent{\bf Case (2)} To prove statement (2), we consider the signature function evaluated at $t = \frac{l-2}{l} \in X_l$.  Using our notation from the previous section and the fact that $k = ql$, we have
\begin{eqnarray*}
 \sigma_{a,k,b,l}(\tfrac{l-2}{l}) &=& a( \lfloor k (\tfrac{l-2}{l}) \rfloor_k +1) - b (\lfloor l (\tfrac{l-2}{l}) \rfloor_l +1)\\
 &=&a \lfloor q (l-2)  \rfloor_{ql}   - b \lfloor   l-2  \rfloor_l +a - b\\
 &=& aq(l-2) - b(l-2) +a -b\\
  &=& a(k - 2q+1) - b(l-1)
 \end{eqnarray*}
 which equals the formula given in the statement of Lemma~\ref{lith2lemma}, modulo a change of sign to form the absolute value.
 Therefore we have
 $$ \mu(a,k,b,l) = \max_{t\in X_k\cup X_l}( |\sigma_{a,k,b,l} (t)| ) \geq -a(k - 2q+1) + b(l-1),$$
 as desired.

 \vskip.05in

\noindent{\bf Case (3)} The proof of statement (3) is more complicated. Once again, we are looking at the maximum value of $|\sigma_{a,k,b,l}(t)|$ for $t \in X_k \cup X_l$. Since $\sigma_{a,k,b,l}(t)$ can increase only at the points in $X_k$ and can decrease only at the points in $X_l$, it follows that the maximum value of $\sigma_{a,k,b,l} (t)$ is attained at  a point in $X_k$ and the minimum value is attained at a point in $X_l$. We break our argument into two parts, studying $t\in X_k$ and $t \in X_l$ separately. \vskip.05in

\noindent{\bf Case (3a)} Let us begin with $t \in X_l$; that is, $t = j/l$, where $j \equiv l \mod 2$ and $0\le j <l$.  Then, since $j \equiv l \mod 2$, we have $$k = ql+r \equiv qj +r \mod 2.$$
Using this fact, we have the following string of equalities.
\begin{equation*}
\begin{split}
-\sigma_{a,k,b,l}(j/l) & = -a(\lfloor qj +jr /l\rfloor_k +1) +b(j+1)\\
&=-a(qj + \lfloor jr/l\rfloor_r +1) +b(j+1)\\
&=(b-qa)(j+1) - a(\lfloor jr/l\rfloor_r +1) +qa\\
&=\sigma_{b-qa,l,a,r}(j/l) +qa.
\end{split}
\end{equation*}
Of these, the second line, in which we switch from $\lfloor\cdot \rfloor_k$ to $\lfloor\cdot \rfloor_r$, is not immediate.  This equality  follows most easily by considering the cases of $k \equiv r \mod 2$ (in which case $qj$ is even and the equality is immediate) and $k \not\equiv r \mod2$ (in which case we use  $\lfloor A + x \rfloor_k = A +  \lfloor x\rfloor_r$ for $A$ an odd integer). 

It now follows that $$\mu(a,k,b,l)=   \max_{t\in X_k\cup X_l}( |\sigma_{a,k,b,l} (t)| ) \geq  \max_{t\in X_l}( \sigma_{b-qa,l,a,r}(t)) +qa.$$
\vskip.05in

\noindent{\bf Case (3b)} We now consider the case $t \in X_k$. That is, $t = i/k$, where $i \equiv k \mod2$ and $0\le i < k$.  We will choose a particular value for $i$. Namely, let $j \equiv r \mod2$, $0\le j <r$ and let $$i = q(\lfloor jl/r\rfloor_l +2) +j.$$ 
We first want to verify both that $i \equiv k \mod2$ and that $0\le i < k$.  

First, observe: 
\begin{eqnarray*}
i = q(\lfloor jl/r\rfloor_l +2) +j
&\equiv& q\lfloor jl/r\rfloor_l +j \pmod{2}\\
&\equiv& ql +j \pmod{2}\\
&\equiv& ql +r  \pmod{2}\\
&=& k.
\end{eqnarray*}
Thus, $i \equiv k \pmod{2}$. 

Next for any positive number $x$, observe that $\lfloor x \rfloor_l \geq -1$. Therefore $$i = q(\lfloor jl/r\rfloor_l +2) +j \geq 0.$$
Lastly, we want to show that $i < k$. Observe first that $\lfloor jl/r\rfloor_l \le l-2$, since $j < r$.  Thus, $i \le ql +j$.  Again using the fact that $j<r$, this gives $i < ql +r = k$. \vskip.05in

Now we want to compute $\sigma_{a,k,b,l}(i/k)$. We begin as follows:
$$\sigma_{a,k,b,l}(i/k) = a(i+1) - b(\lfloor il/k \rfloor_l +1). $$
We claim that $\lfloor il/k \rfloor_l = \lfloor jl/r\rfloor_l $. Assuming this claim for now and substituting the chosen value for $i$, we can finish off the computation of $\sigma_{a,k,b,l}(i/k)$ as follows: 

\begin{equation*}
\begin{split}
\sigma_{a,k,b,l}(i/k) &= a(i+1) - b(\lfloor il/k \rfloor_l +1) \\
& = a(q\lfloor jl/r \rfloor_l + 2q + j +1) - b(\lfloor j l/r\rfloor_l +1)\\
&= -(b-qa)(\lfloor jl/r \rfloor_l +1) + a(j+1) +qa\\
&= -\sigma_{b-qa,l,a,r}(j/r) +qa.
\end{split}
\end{equation*}
We conclude that $$\mu(a,k,b,l) =   \max_{t\in X_k\cup X_l}( |\sigma_{a,k,b,l} (t)| ) \ge \max_{t \in X_r}(-\sigma_{b-qa, l,a,r}(t)) +qa.$$ 

Putting the conclusions of Cases (3a) and (3b) together, we have:
\begin{eqnarray*}
\mu(a,k,b,l) &\ge& \max_{t \in X_l \cup X_r}(|\sigma_{b-qa, l,a,r}(t)|) +qa\\
&=& \mu(b-qa, l,a,r) + qa,
\end{eqnarray*}
as desired.

Thus it only remains to show that $\lfloor il/k \rfloor_l = \lfloor jl/r\rfloor_l $.   First, consider the fraction $il/k$.

\begin{equation*}
\begin{split}
il/k
&= \tfrac{1}{k} [ql(\lfloor jl/r\rfloor_l +2) +jl ]\\
&= \tfrac{1}{k} [(k-r)(\lfloor jl/r\rfloor_l +2) +jl ]\\
&=\lfloor jl/r \rfloor_l +2 - \tfrac{1}{k}[r(\lfloor jl/r \rfloor_l +2) - jl]\\
\end{split}
\end{equation*}
Now the right hand side of the equation has become complicated, but we will see that the value of the expression $ \tfrac{1}{k}[r(\lfloor jl/r \rfloor_l +2) - jl]$ is small. Consider the following basic fact:
$$ jl/r < \lfloor jl /r \rfloor_l +2 \le jl/r +2.$$
From this, we have $$ 0 < r(\lfloor jl/r \rfloor_l +2) - jl \le 2r.$$
Moreover, since $r <k$, it follows that $$   0 < \tfrac{1}{k}[r(\lfloor jl/r \rfloor_l +2) - jl]  < 2.$$ Returning to our previous computation and using this result, we have

\begin{equation*}
\begin{split}
il/k   &= \lfloor jl/r \rfloor_l +2 - \tfrac{1}{k}[r(\lfloor jl/r \rfloor_l +2) - jl]\\
&=  \lfloor jl/r \rfloor_l  + \epsilon,
\end{split}
\end{equation*}
where $\epsilon \in (0,2)$. It follows that  $\lfloor il/k \rfloor_l = \lfloor jl/r\rfloor_l $, as desired.


\section{A  family in which Upsilon determines four-genus.}\label{section-purely-upsilon}
In this section we present an infinite family of knots for which we can realize the lower bound on the four-genus that arises from Upsilon.  We are also able to completely determine values of $t$ for which $g_4(K) = |\U_K(t)/t|$ and observe how that maximizer $t$ depends on $a$ and $b$.
 
We first need some observations about the function $\U_{T(p,q)}(t)$. In general, the function $\U_{T(p,q)}(t)$ is determined by an inductive formula in~\cite[Theorem 1.15]{oss}. In the special cases $T(p, p+1)$ and $T(2,q)$, the functions have been (in part or in whole) computed concretely by Ozsv\'ath, Stipsicz, and Szab\'o~\cite{oss}. In addition, Feller~\cite{feller2} explicitly determined the function $\U_{T(p,q)}(t)$ for the cases $p = 3$ or $4$. For our purposes, we consider the general case $\U_{T(p,q)}(t)$ and determine the value of the function through the first two singularities of the function. 

\begin{utheorem}[Proved in Appendix A]
Consider the torus knot $T(p,q)$ where $p<q$. We write $q = kp+d$ where $0< d <p$. The first singularity of $\U_{T(p,q)}(t)$ is at $t_1 = \tfrac{2}{p}$, and the second singularity is at $t_2 = \tfrac{4}{p}$ if $d \leq \tfrac{p}{2}$ and at $t_2 = \tfrac{2}{d}$ if $d\geq \tfrac{p}{2}$.  Moreover, the values of $\U_{T(p,q)}(t)$ on the interval $[0,t_2]$ are as follows:

$$\U_{T(p,q)}(t)  = \begin{cases}
-\tfrac{1}{2}(p-1)(q-1)t & \text{ for all } t\in[0,\tfrac{2}{p}]\\
 - \left[\tfrac{1}{2}(p-1)(q-1) - kp\right]t - 2k & \text{ for all } t\in[\tfrac{2}{p}, t_2]
 \end{cases}
 $$
\end{utheorem}
\vspace{.3cm}
 
Using the above result, we have the following computations:
 
\begin{lemma}\label{computation} Let $K= aT(p,qr)\cs -bT(q,pr)$, where $a , b>0$, $p < q$, and $r < \frac{q}{q-p}$.
\begin{itemize}
\item  $|\tau(K)| =| \frac{a}{2}(p-1)(qr-1) - \frac{b}{2}((q-1)(pr-1)|$.\vskip.05in
\item $ |\frac{p}{2}\cdot\U_K(\frac{2}{p})| =  | \frac{a}{2}(p-1)(qr-1) -  b[\tfrac{1}{2}(q-1)(pr-1) +  (p-q)(r-1)]|. $\vskip.05in
\end{itemize}
\end{lemma}
\begin{proof}

For a torus knot $T(p,q)$, the value of $\tau$ is computed as follows: $\tau(T(p,q)) = \tfrac{1}{2}(p-1)(q-1)$. Using this and the fact that $\tau$ is additive over connected sum gives the first equality for $|\tau(K)|$.

The computation of $ |\frac{p}{2}\cdot\U_K(\frac{2}{p})|$ takes more care. By the additivity of $\Upsilon_K$ over connected sum, we know that
$$ |\tfrac{p}{2}\cdot\U_K(\tfrac{2}{p})| = \tfrac{p}{2} |a\U_{T(p,qr)}(\tfrac{2}{p}) - b\U_{T(q,pr)}(\tfrac{2}{p})|.$$
Therefore, it suffices to compute $\U_{T(p,qr)}(\tfrac{2}{p})$ and $\U_{T(q,pr)}(\tfrac{2}{p})$ separately. 

From the previous theorem, it follows that 
 \begin{eqnarray*}
 \U_{T(p,qr)}(\tfrac{2}{p}) &=& -\tfrac{1}{2}(p-1)(qr-1)\tfrac{2}{p}\\
 & =& -\tfrac{1}{p}(p-1)(qr-1).
 \end{eqnarray*}

Next, we consider $\U_{T(q,pr)}(t)$. Again, we would like to evaluate this function at $t = \tfrac{2}{p}$. We show that  regardless of which of the two possible values the second singularity $t_2$ is, the value $\tfrac{2}{p}$ is contained in the interval $[\tfrac{2}{q}, t_2]$.

First, suppose the second singularity $t_2$ is at $\tfrac{4}{q}$. We show that $\tfrac{2}{p} \leq \tfrac{4}{q}$. This inequality is equivalent to $q \leq 2p$.  However, if $q > 2p$, then $p < \tfrac{q}{2}$, so  $r< \tfrac{q}{q-p} < \tfrac{q}{q - q/2} = 2$, which is impossible. Therefore $\tfrac{2}{p} \in [\tfrac{2}{q}, \tfrac{4}{q}]$.

Now suppose that the singularity $t_2$ is the other alternative, namely,  $t_2 = \tfrac{2}{d}$, where $d$ is computed as follows. Our assumptions that $p < q$ and $r < \tfrac{q}{q-p}$ imply that $(r-1)q < pr< qr$. Therefore, in terms of the notation in the previous theorem, $k = r-1$ and $d = pr - (r-1)q$.  So we want to show that $\tfrac{2}{p} \leq \tfrac{2}{d}$. This is equivalent to $d \leq p$. But this inequality is an immediate consequence of $p< q$. Hence, regardless of the value of the second singularity $t_2$, we know that $\tfrac{2}{p} \in [\tfrac{2}{q}, t_2]$, so we can evaluate $\U_{T(q,pr)}(\tfrac{2}{p})$ using the formula in the theorem above. 

$$\U_{T(q,pr)}(\tfrac{2}{p}) = -\left[\tfrac{1}{2}(q-1)(pr-1) - (r-1)q\right](\tfrac{2}{p}) - 2(r-1).$$

What remains is an algebraic manipulation:

\begin{eqnarray*}
|\tfrac{p}{2}\cdot\U_K(\tfrac{2}{p})| &=& \tfrac{p}{2} |a\U_{T(p,qr)}(\tfrac{2}{p}) - b\U_{T(q,pr)}(\tfrac{2}{p})|\\
&=& |-\tfrac{1}{2}a(p-1)(qr-1)  + b\left[\tfrac{1}{2}(q-1)(pr-1) - (r-1)q + p(r-1)\right]|\\
&=& |\tfrac{1}{2}a(p-1)(qr-1)  - b\left[\tfrac{1}{2}(q-1)(pr-1) + (p-q)(r-1) \right]|.
\end{eqnarray*}

\end{proof}

The above computations are precisely what we need to determine the 4-genus of this family of knots. 

\begin{theorem}  Let $K = aT(p,qr)\cs  -bT(q,pr)$ where $a, b >0$, $0 < p < q$, and $r< \tfrac{q}{q-p}$.  Then
$$g_4(K) = \begin{cases} 
|\tau(K)| &\mbox{if }  a\leq b \\
 | \tfrac{p}{2} \U_K(\tfrac{2}{p})|& \mbox{if } a\geq b. \\
 
\end{cases} 
$$
Furthermore, if $a= b$, then $  \tau(K) = -  \frac{p}{2} | \U_K(\frac{2}{p})|$; if $a >b$, then $  |\tau(K)| <  \frac{p}{2} | \U_K(\frac{2}{p})|$.
\end{theorem}

\begin{proof}
First of all, the last line of the statement of the theorem follows immediately from Lemma~\ref{computation}. 

Before we face the rest of the theorem's statement, we make note of a result of Baader~\cite[Proposition 1]{baader1} that implies that the fiber surface for $T(q,pr)$ contains a fiber surface for $T(p,qr)$.

Thus, for the case $a \le b$, starting with the canonical Seifert surface for $K$, we can cut out the fiber surface for $aT(p,qr) \cs -aT(p,qr)$ and replace it with a slice disk. This creates a surface of genus $$\tfrac{a}{2}(p-1)(qr-1) +\tfrac{b}{2}(q-1)(pr-1) - 2 (\tfrac{a}{2}(p-1)(qr-1)).$$  This simplifies to $$\tfrac{b}{2}(q-1)(pr-1) - \tfrac{a}{2}(p-1)(qr-1),$$  which equals $|\tau(K)|$(as stated in Lemma~\ref{computation}).

For the case $a\geq b$, we again start with the canonical Seifert surface for $K$. Using the result of Baader, we can cut out a Seifert surface for $bT(p,qr) \cs -bT(p,qr)$ and replace it with a slice disk. This creates a surface of genus 
$$\tfrac{a}{2}(p-1)(qr-1) +\tfrac{b}{2}(q-1)(pr-1) - 2 (\tfrac{b}{2}(p-1)(qr-1)),$$  
which simplifies to equal $ |\frac{p}{2}\cdot \U_K(\frac{2}{p})| $ (as stated in Lemma~\ref{computation}).

\end{proof}


\section{Mixed example}\label{mixed-example}
Now we consider a single example, described in the introduction as Theorem~\ref{example3438}, consisting of linear combinations    $aT(3,4) \# - bT(3,8)$.  These have the  interesting property that we can always realize the best lower bound on the four-genus of $aT(3,4) \cs -bT(3,8)$ that arises from signatures and Upsilon; however, for $a \le 2b$, the best bound comes from the Upsilon function, whereas for $a \ge 2b$, the best bound comes from  the signature function.  

\begin{proof}[Proof of Theorem~\ref{example3438}]
Let $K = aT(3,4) \cs -bT(3,8)$.  The lower bounds on four-genus coming from the Upsilon function and the signature function are as follows (for references on how this is done, see~\cite{feller2, goldsmith, litherland, oss}): 

$\textrm{max}_{t \in (0,1]}\{\frac{1}{t}|\Upsilon_{K}(t)|\} =  
\begin{cases} 
{\bf 7b-3a}&{\bf \mbox{\bf if } 0\leq a< 2b} \\
5b-2a& \mbox{if } 2b \leq a < \frac{12}{5}b \\
3a - 7b & \mbox{if } a \geq  \frac{12}{5} b,
\end{cases}  $\\

\vspace{.5cm}

$\textrm{max}_{t \in [0,1]}\{\tfrac{1}{2}\sigma_{K}(t)\} =  
\begin{cases} 
6b-3a &\mbox{if } 0\leq a< b \\
5b - 2a & \mbox{if } b \leq a < 2b \\
{\bf 3a - 5b} & \mbox{\bf if }  {\bf a \geq 2b}.
\end{cases}  $\\

One can observe that the Upsilon bound is stronger than the signature bound for $0\leq a\leq 2b$, and then for $a\geq 2b$, the signature bound is the stronger of the two. 

Now suppose that $0\leq a \leq 2b$. The canonical Seifert surface for $K$ has genus $3a + 7b$. This surface contains a subsurface which has boundary equal to $aT(3,4)\cs - aT(3,4)$. Since this knot is slice, we can cut out this subsurface and glue in a disk in $B^4$  in its place. The subsurface which we removed had genus $6a$. Therefore our newly formed surface has genus $7b - 3a$, and so it follows that $g_4(K) \leq 7b-3a$. Combining this with the Upsilon lower bound, we obtain $g_4(K) = 7b-3a$ for $0\leq a \leq 2b$, as desired.

Next, suppose that $a\geq 2b$. Once again we begin with the canonical Seifert surface for $K$. This surface contains a subsurface which has boundary $2bT(3,4) \cs -2bT(3,4)$. Cut out this subsurface and glue in  a disk in its place. The subsurface which we removed had genus $12b$. Therefore our newly formed surface has genus $3a - 5b$, and it follows that $g_4(K) \leq 3a-5b$ for $a \geq 2b$. Combining this with the signature lower bound on four-genus, we conclude that $g_4(K) = 3a-5b$ for  $a \geq 2b$. 
\end{proof}

Figure~\ref{fig:3438} in the introduction illustrates the lower bounds on the stable four-genus of $tT(3,4) \cs -(1-t)T(3,8)$ for $t \in [0,1]$; the thick (blue) line represents the lower bound obtained from $\U_K(t)$ and the thin (red) line represents the bound obtained from $\sigma_K(t)$.

Similar work shows that the families arising from $T(3,5)$ and either $T(3,10), T(3,20)$, or $T(3,25)$, have the same property as the above example -- namely, the four-genus of all linear combinations is determined using both the Upsilon and signature lower bounds together.

\section{Generalizing to $K = aT(p,q) \cs -bT(p',q')$}\label{generalizing}

In spite of the success detailed thus far in using signature and Upsilon invariants, we do not have to look far to find examples where these invariants and our best geometric realizations are not sufficient to compute the four-genus. Since Theorem~\ref{thm:litherland} resolves all knots of the form $aT(2,k) \cs -bT(2,l)$, the first possible examples of our limitations would be knots of the form $aT(2,k) \cs -bT(3,l)$ or $aT(2,k) \cs -bT(4,l)$. Indeed, though the Upsilon function resolves a subset of these knots, many unknown cases remain. 

\subsection{Special cases of $aT(2,k) \cs -b T(3,l)$ and $aT(2,k) \cs -b T(4,l)$}
We begin by detailing the knots for which our methods remain sufficient. 

 \begin{theorem}\label{subfamily}
  Let   $ r \ge  1$, and let $K$ be any of the following knots:
 \begin{enumerate}
\item $ aT(2,10r+1)\cs-bT(3,6r+1) $.\vskip.05in

\item $aT(2,10r+3)\cs-bT(3,6r+2) $.\vskip.05in

\item $aT(2,10r+1)\cs-bT(4,4r+1)$.\vskip.05in
\end{enumerate}
Then the 4-genus of $K$ is equal to $|\tau(K)|$ for $a \le b$ and equals $|\U_K(1)|$ for $a\ge b$.
\end{theorem}

\begin{proof}
We begin with the first case.  Let $K = aT(2,10r+1)\cs-bT(3,6r+1)$. First, we compute $\tau(K)$ and $\U_K(1)$. (To compute $\U_K(1)$, we use~\cite[Proposition 28]{feller2}.)
$$ \tau(K) = 5ar - 6br.$$
$$\U_K(1) = -5ar + 4br.$$

Both of these expressions provide a lower bound for $g_4(K)$. Now we construct surfaces to realize these lower bounds. From work of Feller~\cite[Proposition 22]{feller2}, we know that the fiber surface for $T(3,m)$ contains a fiber surface for $T(2,n)$ as long as $n\leq \frac{5m-1}{3}$. Thus, it follows that the standard Seifert surface for $T(3, 6r+1)$ contains a subsurface which has boundary $T(2, 10r+1)$.

If $a\leq b$, then the standard Seifert surface for $K$ contains a subsurface which has boundary $aT(2, 10r+1) \# -aT(2,10r+1)$. Cut out this subsurface and glue a slice disk in its place. The original Seifert surface had genus $5ar + 6br$, and the subsurface which we removed had genus $10ar$. Therefore our newly formed surface has genus $6br - 5ar$, which equals $ |\tau(K)|$.

On the other hand, suppose that $a\geq b$. Then the standard Seifert surface for $K$ contains a subsurface which has boundary $bT(2, 10r+1) \# -bT(2,10r+1)$. Cutting out this subsurface and gluing in a slice disk, we obtain a surface with genus $5ar - 4br$, which equals $|\U_K(1)|$.

The remaining two cases proceed in exactly the same way, using~\cite[Proposition 28]{feller2} to compute $\U_K(1)$ and using ~\cite[Propositions 22 and 23]{feller2} to construct the desired surfaces. We note that the values of the invariants in each of the  remaining cases are as follows:

Let $K  = aT(2,10r+3)\cs-bT(3,6r+2))$. 
$$\tau(K) = a(5r+1) - b(6r+1).$$
 $$\U_K(1) = -a(5r+1) + b(4r+1).$$
 
Let $K = aT(2,10r+1)\cs-bT(4,4r+1)$.
$$\tau(K) = 5ar - 6br.$$
$$\U_K(1) = -5ar+4br.$$
\end{proof}

\subsection{Open problem} 
Allison Miller informs us that she has determined the four-genus of connected sums $aT(2,q) \# -bT(3,q')$ for small values of $q$ and $q'$.  Her examples include $q = 5$, $q' = 7$, which was given as an open problem in~\cite{livingston1}.  In each of the cases that she resolved, the four-genus was determined by the signature function and $\tau$.   

Peter Feller informs us that he has shown that $$g_4(aT(2,3)\#-bT(3,4))=\max\{3b-a,b,a-2b\},$$ realizing the lower bound that is given by the signature function.

The simplest example that we have found in which signatures and $\tau$ do not suffice is the combination  $K = aT(2,13) \#-  bT(3,4)$.  In this example, for $ a < \tfrac{1}{2} b$, the signature bound is stronger than that provided by $\U_K(t)$.  For $ a > \tfrac{1}{2} b$,   $\U_K(1)$ provides a stronger bound on the four-genus than does either the signature or $\tau$.  In particular, we have the following lower bounds: 

$$ g_4(K) \geq 
\begin{cases}
3b-5a    & \mbox{\ if\ } 0 \leq a < \tfrac{1}{3} b \\
 2b-2a & \mbox{\ if\ }\tfrac{1}{3} b \leq a < \tfrac{1}{2} b \\
 6a-2b & \mbox{\ if\ } a > \tfrac{1}{2} b.
 \end{cases}
   $$

In Figure~\ref{T2-13-3-4.pdf}, we graph the relevant bounds in terms of the stable four-genus.  The bound arising from $\U_K(t)$ is represented by the solid (red) graph, the signature bound is represented by the dotted (black) graph, and the bound arising from $\tau(K)$ is represented by the dashed (blue) graph.

\begin{figure}[h]
\fig{.45}{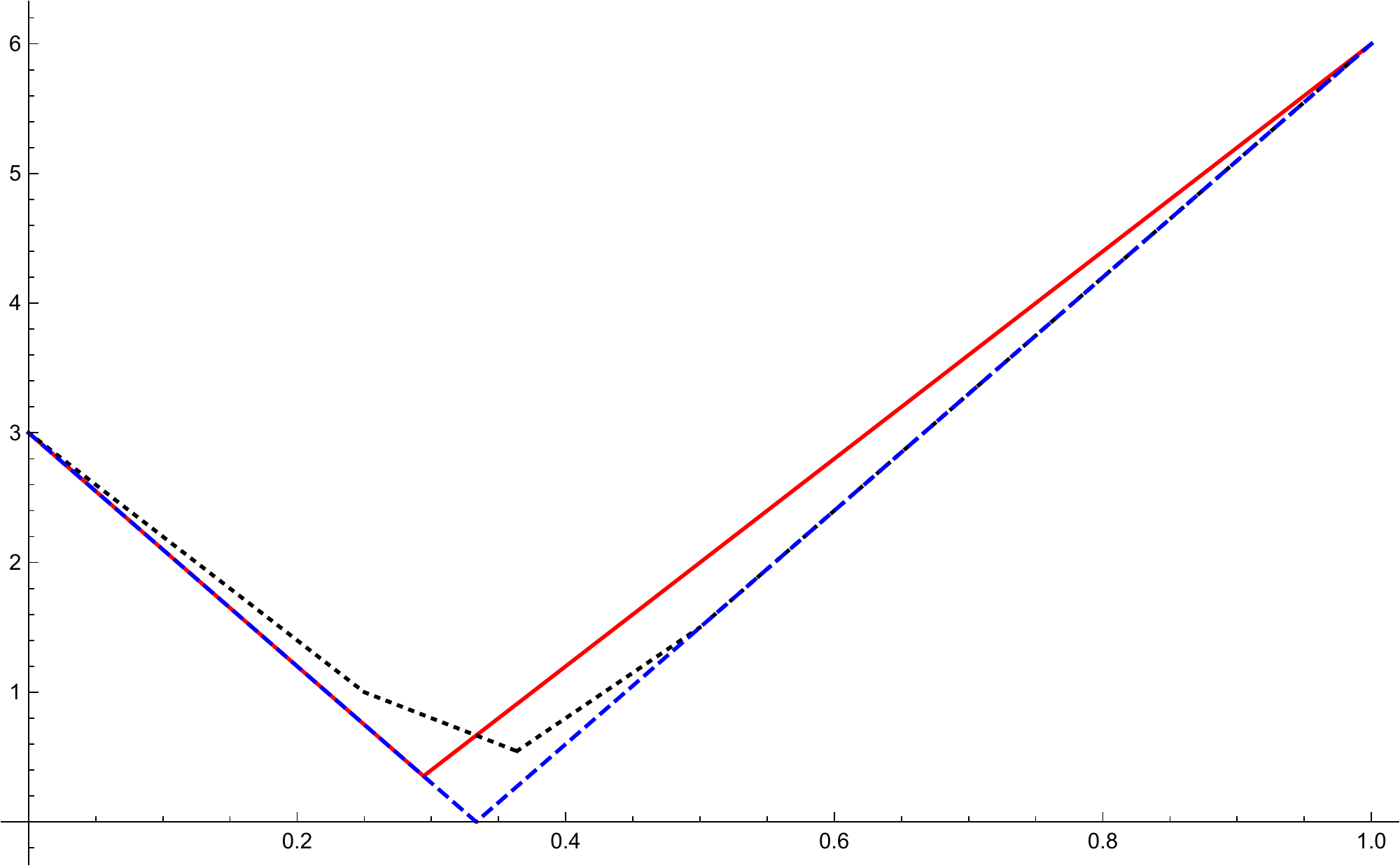}  \caption{Lower bounds on the stable four-genus of $K = tT(2,13) \# -(1-t) T(3,4)$ arising from $\U_K(t)$ is represented by the solid (red) graph, the signature bound is represented by the dotted (black) graph, and the bound arising from $\tau(K)$ is represented by the dashed (blue) graph. For $t\leq \tfrac{1}{3}$, signature is the strongest lower bound, while the Upsilon bound is the strongest bound for $t \geq \tfrac{1}{3} $. }\label{T2-13-3-4.pdf}
\end{figure}

We are unable to determine the exact four-genus for all knots in this family using our usual methods of geometric realization. The best we can do is as follows:

$$ g_4(K) \leq
\begin{cases}
3b-4a    & \mbox{\ if\ } 0 \leq a < \tfrac{1}{3} b \\
 2a+b & \mbox{\ if\ }\tfrac{1}{3} b \leq a < \tfrac{1}{2} b \\
 6a-b & \mbox{\ if\ } a \geq \tfrac{1}{2} b.
 \end{cases}
   $$

A similar story repeats itself with other knots of the form $aT(2,k) \cs -bT(3,l)$ or $aT(2,k) \cs -bT(4,l)$ which are not covered by Theorem~\ref{subfamily}. We can determine the four-genus for some values of $a$ and $b$, but not all.  (As mentioned above, Peter Feller recently determined that in the first case, $k=3, l=4$, the lower bound given by the signature function can be realized.)


\appendix
\section{Staircases and the Upsilon function for torus knots}\label{appendix:torusknotU}

In~\cite{borodzik-livingston}, a different perspective on the Upsilon function for the torus knot $T(p,q)$ is discussed. The function is  computed by first creating a so-called staircase and then minimizing an expression over all points in the staircase. This process is summarized in this appendix, and we use this approach to compute the value of $\Upsilon_{T(p,q)}(t)$ up to its second singularity. 

Given an increasing sequence of integers, $\{0=x_0, x_1, \ldots , x_n = N\}$, with $N $ even, we recursively  define a sequence of points in the plane,  $A_i = (a^i_1,a^i_2)$,  $0\le i \le 2n$:\vskip.05in
\begin{itemize}
\item  $A_0 = (a^0_1, a^0_2) =  (0,N/2)$, \vskip.05in

\item $A_{2k+1} =  (a^{2k}_1 + 1, a^{2k}_2),  $ \vskip.05in

\item $A_{2k+2} =  (a^{2k+1}_1 , a^{2k+1}_2 - (x_{k+1} - x_{k} -1)) $. \vskip.05in

\end{itemize}
\vskip.05in 
Here are three examples:

\noindent  $$S_1 = \{0,3,5,6,8\}\rightarrow \{(0,4), (1,4), (1,2), (2,2), (2,1), (3,1), (3,1), (4,1), (4,0)\}.$$
 
 \noindent  $$S_2 = \{ 0,5,6,8\} \rightarrow\{(0,4), (1,4), (1,0), (2,0), (2,0),(3,0), (3,-1)\}.$$ 

\noindent  $$S_2' = \{ 0,3,6,8\}\rightarrow \{(0,4), (1,4), (1,2), (2,2), (2,0), (3,0),(3,-1)\}.$$ 

Sequences of points constructed in this way are called {\em staircases}. Duplicate pairs of points can be deleted from the sequence without affecting the results of our computations, so we do so.  In the previous examples, this yields the three staircases  illustrated in Figure~\ref{fig:staircases}.

\begin{figure}[h]
\fig{.22}{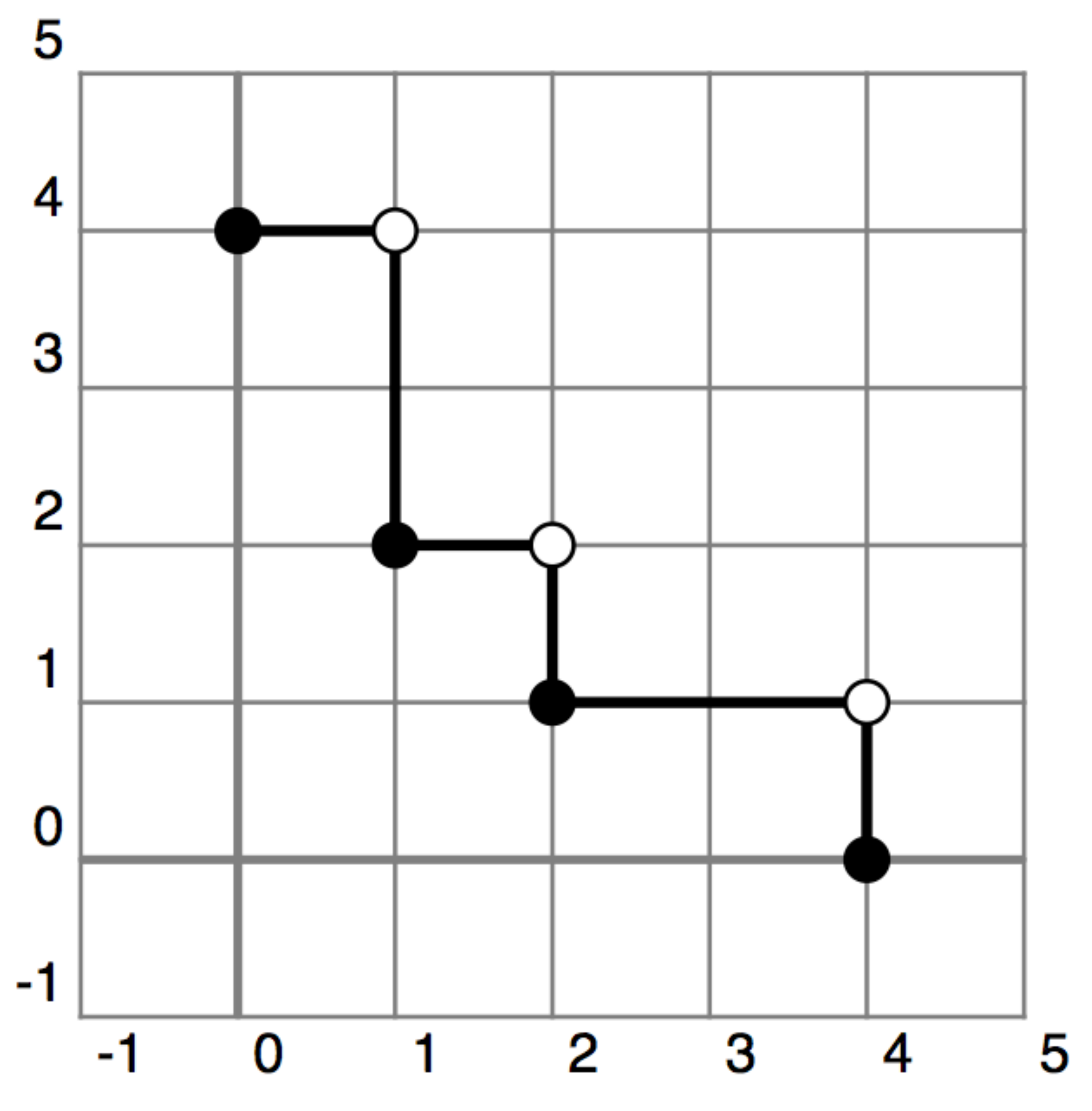} \hskip.05in\fig{.22}{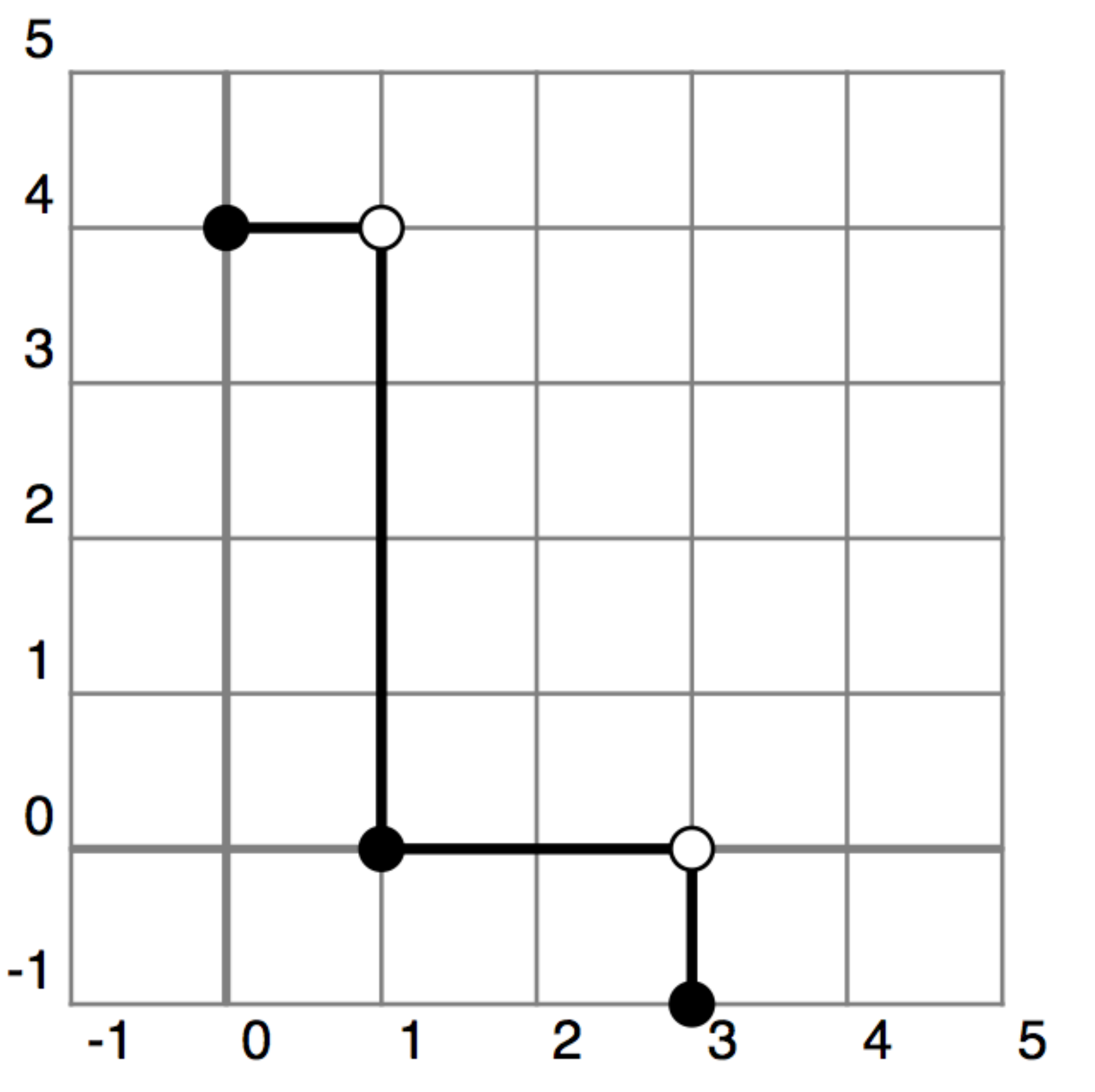} \hskip.05in
\fig{.22}{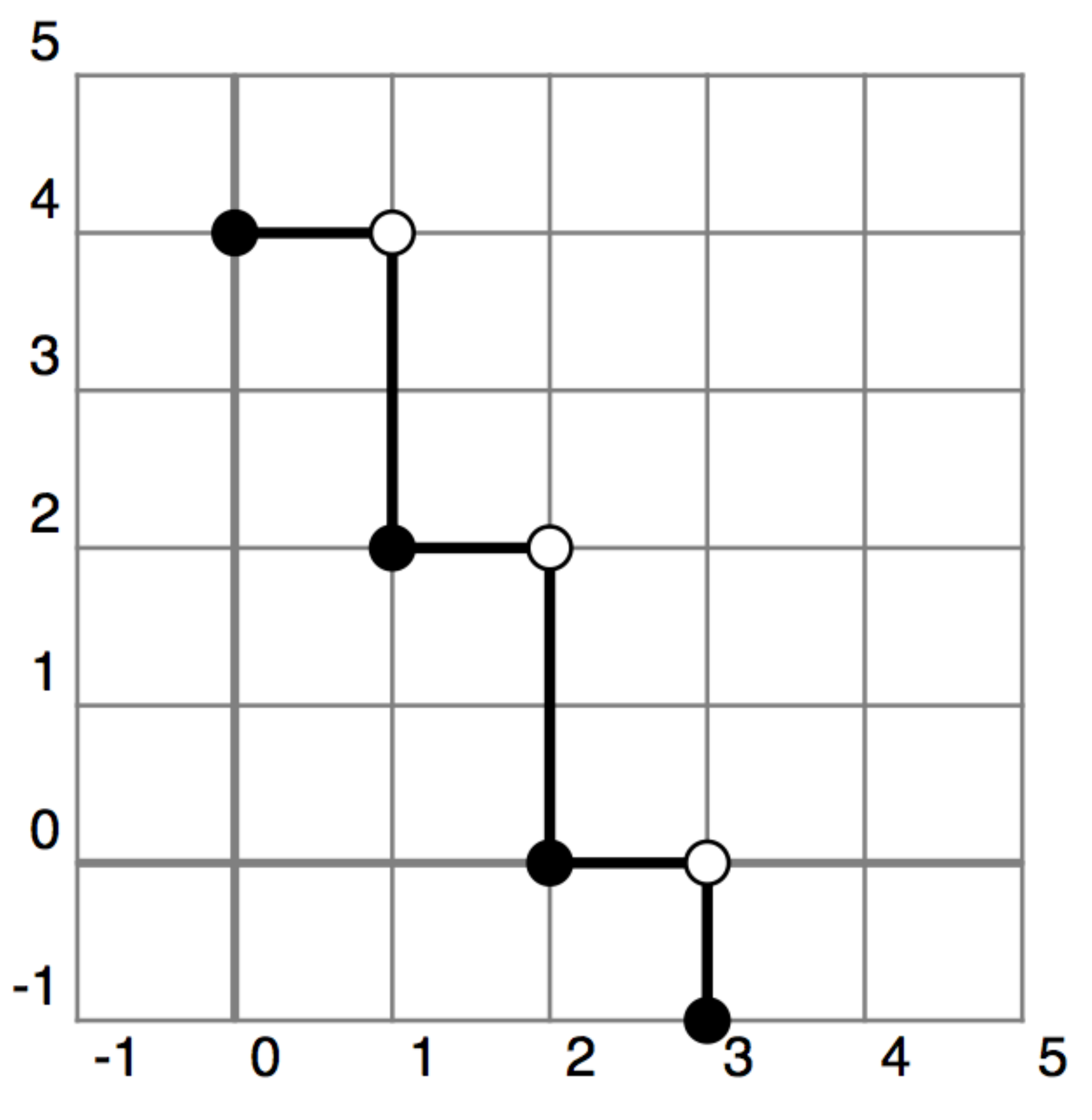}
 \caption{Three examples of staircases }\label{fig:staircases}
\end{figure}

For any staircase as above, we define $$U_S(t) = \min_{k\in \{0,\ldots, n\}} \{(1 - \tfrac{t}{2})a^{2k}_1 +\tfrac{t}{2}a^{2k}_2\}.$$  For any $t \in [0,1]$, $U_S(t)$ can be thought of as the least value of $B$ such that the line $(1- \tfrac{t}{2})x + \tfrac{t}{2}y = B$ contains one of the points $A_{2k}$.

In the examples above, observe that $S_2 \subset S_1$ and $S_2' \subset S_1$. When one sequence is a subsequence of another, the associated functions have the following relationship:

\begin{lemma}\label{contained} Let $S_1$ and $S_2$ be increasing sequences of integers. If $S_2 \subset S_1$, then $U_{S_1}(t) \ge U_{S_2}(t)$ for all $t$.
\end{lemma}
\begin{proof}  This is proved inductively, observing the effect of adding a single element to the sequence $S_2$.  In general, if all elements of the staircase associated to a sequence lie on or above a line $\mathcal L$, then after adding elements to the sequence, all points on the new staircase will also lie on or above the same line.  Showing this is an elementary exercise. Details are left to the reader.
\end{proof}

According to~\cite{borodzik-livingston}, for the torus knot $T(p,q)$ with $p<q$, the Upsilon function is given by $\U_{T(p,q)}(t) = -2U_{S(p,q)}(t)$, where $S(p,q)$ is the semigroup generated by $p$ and $q$ truncated at $(p-1)(q-1)$.  As an example,   $S(3,5) =  \{0,3,5,6,8\}$.

We have the following theorem about $\U_{T(p,q)}(t)$.

\begin{theorem}
Consider the torus knot $T(p,q)$ where $p<q$. We write $q = kp+d$ where $0< d <p$. The first singularity of $\U_{T(p,q)}(t)$ is at $t_1 = \tfrac{2}{p}$, and the second singularity is at $t_2 = \tfrac{4}{p}$ if $d\leq \tfrac{p}{2}$ and at $t_2 = \tfrac{2}{d}$ if $d\geq \tfrac{p}{2}$.  Moreover, the values of $\U_{T(p,q)}(t)$ on the interval $[0,t_2]$ are as follows:

$$\U_{T(p,q)}(t)  = \begin{cases}
-\tfrac{1}{2}(p-1)(q-1)t & \text{ for all } t\in[0,\tfrac{2}{p}]\\
 - \left[\tfrac{1}{2}(p-1)(q-1) - kp\right]t - 2k & \text{ for all } t\in[\tfrac{2}{p}, t_2]
 \end{cases}
 $$
\end{theorem}

\vspace{.3cm}
\begin{proof} The sequence  $S(p,q)$ contains a subsequence:
$$S= \{0, p, 2p, \ldots, kp, q = kp+d , (k+1)p, (k+1)p+d, (k+2)p, (k+2)p + d, \cdots\}.$$

This sequence is regular enough that the process of constructing the corresponding staircase and computing the function $U_S(t)$ is fairly straightforward.  The stairs in the staircase have height $p-1$ for the first several stairs, and then the heights later alternate between $d-1$ and $p-d-1$. The functional values and singularities of $U_{S}(t)$ are identical to those described for $\U$ in the statement of the theorem.

The sequence $S(p,q)$ is constructed from $S$ by including elements, all greater than $(k+1)p$. As we observed in Lemma~\ref{contained}, since $S\subset S(p,q)$, it follows that $U_{S(p,q)}(t)\geq U_S(t)$. A bit more care shows that since the elements we are adding to sequence are greater than $(k+1)p$, the function $U_S(t)$ is not affected for small $t$.  Details are left to the reader.
\end{proof}



\end{document}